\newcommand{\incs}{
  \usepackage{amsmath, amsthm, amsfonts, hyperref}

\theoremstyle{plain}

\newtheorem{thm}{Theorem}
\newtheorem{lem}[thm]{Lemma}
\newtheorem{cor}[thm]{Corollary}

\theoremstyle{definition}
\newtheorem{defn}[thm]{Definition}
\newtheorem{exl}[thm]{Example}

\renewcommand{\phi}{\varphi}
\renewcommand{\bar}{\overline}

\DeclareMathOperator{\Eq}{Eq}
\DeclareMathOperator{\Fix}{Fix}
\DeclareMathOperator{\Rem}{Rem}
\DeclareMathOperator{\id}{id}

}

\newcommand{\abs}{
\begin{abstract}
We give an easily checkable algebraic condition which implies that two
elements of 
a finitely generated free group are members of distinct
doubly-twisted conjugacy classes with respect to a
pair of homomorphisms. We further show that this criterion is
satisfied with probability 1 when the homomorphisms and elements are
chosen at random.  
\end{abstract}
}

\documentclass{article}

\incs
\begin{document}

\bibliographystyle{hplain}

\title{Typical elements in free groups are in different doubly-twisted conjugacy classes}
\author{P. Christopher Staecker
\thanks{Email: cstaecker@fairfield.edu}
}

\maketitle
\abs

\section{Introduction}
Let $G$ and $H$ be finitely generated free groups, and $\phi, \psi: G \to H$ be
homomorphisms. The group $H$ is partitioned into the set of
\emph{doubly-twisted conjugacy classes} as follows: $u, v \in H$ are in the
same class (we write $[u] = [v]$) if and only if there is some $g\in G$
with
\[ u = \phi(g) v \psi(g)^{-1}. \]

Our principal motivation for studying doubly-twisted conjugacy is Nielsen
coincidence theory (see \cite{gonc05} for a survey), the study of
the coincidence set of a pair of mappings and the minimization of this
set while the mappings are changed by homotopies. Our focus on free
groups is motivated specifically by the problem of computing \emph{Nielsen
classes} of coincidence points for pairs of mappings $f, g: X \to Y$,
where $X$ and $Y$ are compact surfaces with boundary. 

A necessary
condition for two coincidence points to be combined by a homotopy
(thus reducing the total number of coincidence points) is that they
belong to the same Nielsen class. (Much
of this theory is a direct generalization of similar techniques in
fixed point theory, see \cite{jian83}.) The number of ``essential''
Nielsen classes is called the Nielsen number, and is a lower bound for
the minimal number of coincidence points when $f$ and $g$ are allowed
to vary by homotopies. 

In our setting, deciding when two coincidence points are in the
same Nielsen class is equivalent to solving a natural doubly-twisted
conjugacy problem in the fundamental groups, using the induced
homomorphisms given by the pair of mappings. 
Thus the Nielsen classes of coincidence points correspond to
twisted conjugacy classes in $\pi_1(Y)$. 

The problem of computing doubly-twisted conjugacy classes in free groups is
nontrivial, even in the singly-twisted case which arises in fixed
point theory, where $\phi$ is an endomorphism and $\psi$ is the identity.
Existing techniques for computing doubly-twisted conjugacy are
adapted from singly-twisted methods using abelian and nilpotent
quotients \cite{stae07b}. Our
main result is not adapted from singly-twisted methods: it
is suited specifically for doubly-twisted conjugacy and in fact can
never apply in the case where $\psi = \id$.

In Section \ref{mainsection} we will present a
remnant condition which can be used to show that two words are in different
doubly-twisted conjugacy classes. In Section \ref{genericsection} we
show in fact that this 
remnant condition is very common for ``most'' homomorphisms. In the
sense of asymptotic density, we show that if
the homomorphisms $\phi, \psi$, and elements $u,v$ are all chosen at
random, then $[u]\neq [v]$ with probability 1.

The author would like to thank Robert F. Brown, Benjamin Fine, Armando
Martino and Enric Ventura for many helpful comments.

\section{A remnant condition for doubly-twisted conjugacy}\label{mainsection}

Given homomorphisms $\phi, \psi: G \to H$, the
\emph{equalizer subgroup} $\Eq(\phi, \psi) \le G$ is the subgroup
\[ \Eq(\phi, \psi) = \{ g \in G \mid \phi(g) = \psi(g) \}. \]
Our first lemma is an equalizer version of a result for singly-twisted
conjugacy which appears in the proof of Theorem 1.5 of \cite{bmmv06}.

Let $\langle z\rangle$ be the free group generated by $z$, and let
$\hat G = G * \langle z \rangle$ and $\hat H = H * \langle z \rangle$, with $*$
the free product. Let $h^v = v^{-1} h v$, and $\phi^v(g) =
v^{-1}\phi(g)v$. The following lemma holds when $G$ and $H$ are any
groups (not necessarily free):

\begin{lem}\label{doublemv}
Given $u \in H$, let $\hat \phi_u: \hat G \to \hat H$ be the extension
of $\phi$ given by $\hat\phi_u(z) = uzu^{-1}$, and let $\hat\psi: \hat
G \to \hat H$ be the extension of $\psi$ given by $\hat\psi(z) = z$. 

Then $[v]=[u]$ if and only if there is some $g \in G$ with $gzg^{-1} \in
\Eq(\hat \phi_u^v, \hat \psi)$.
\end{lem}
\begin{proof}
First, assume that $[v]=[u]$, and let $g\in G$ be some element with
\[ v = \phi(g)u\psi(g)^{-1}. \]
Then we have
\begin{align*}
\hat\phi_u^v(gzg^{-1}) &= v^{-1}\phi(g)uzu^{-1} \phi(g)^{-1}v \\
&= \psi(g)u^{-1}\phi(g)^{-1} \phi(g) u z u^{-1} \phi(g)^{-1}
\phi(g)u\psi(g)^{-1} \\
&= \hat\psi(g z g^{-1})
\end{align*}
as desired.

Now assume that there is some element $gzg^{-1} \in \Eq(
\hat\phi_u^v, \hat\psi)$ for $g \in G$. Consider the commutator:
\begin{align*}
[u^{-1}\phi(g)^{-1}v\psi(g), z] 
&= u^{-1}\phi(g)^{-1}v\psi(g)z \psi(g)^{-1} v^{-1} \phi(g) u z^{-1} \\
&= u^{-1} \phi(g)^{-1} v \hat\psi(gzg^{-1}) v^{-1} \phi(g) u z^{-1} \\
&= u^{-1} \phi(g)^{-1}v \hat\phi_u^v(gzg^{-1}) v^{-1} \phi(g) u z^{-1}
\\
&= u^{-1} \phi(g)^{-1} v v^{-1} \phi(g) uzu^{-1} \phi(g)^{-1} v v^{-1}
\phi(g) u z^{-1} = 1.
\end{align*}
Thus $z$ commutes with $u^{-1}\phi(g)^{-1}v \psi(g)$, and so
$u^{-1}\phi(g)^{-1}v\psi(g) = 1$, since this word does not contain
the letter $z$. Thus $[u]=[v]$. 
\end{proof}

The above lemma is difficult to apply for the purpose of computing
twisted conjugacy classes, since the problem of computing the equalizer
subgroup of homomorphisms is difficult. In fixed point
theory (where $\psi = \id$), if $\phi$ is an automorphism, an
algorithm of \cite{masl03} is given to compute the fixed point
subgroup $\Fix(\phi)$. This algorithm relies fundamentally on 
the methods of Bestvina and Handel \cite{bh92} for representing
automorphisms of free groups using train tracks, and these techniques
do not extend in an obvious way to coincidence theory. 

Though the equalizer subgroup is in general difficult to compute, we
will show that a certain remnant property will force the equalizer
subgroup to be trivial. Remnant properties were first used by Wagner in
\cite{wagn99}.

\begin{defn}\label{remnantdef}
Let $G$ be a finitely generated free group with a specified set of
generators $\bar G = \{g_1, \dots, g_n\}$. The homomorphism $\phi$
\emph{has remnant} if for each $i$, the word
$\phi(g_i)$ has a nontrivial subword which has no cancellation in
any of the products
\[
\phi(g_j)^{\pm 1} \phi(g_i), \quad \phi(g_i)\phi(g_j)^{\pm 1},
\]
except for $j=i$ with exponent $-1$. The maximal
such noncancelling subword of $\phi(g_i)$ is called the \emph{remnant}
of $g_i$, written $\Rem_\phi (g_i)$.

We will occasionally discuss the length of the remnant
subwords, in one of two ways. 
If, for some natural number $l$, we have $|\Rem_\phi (g_i)| \ge
l$ for all $g_i$, we will say that $\phi$ has \emph{remnant length $l$}.
For some $r \in (0,1)$, we say that $\phi$ has \emph{remnant ratio
$r$} when 
\[ |\Rem_\phi (g_i)| \ge r |\phi(g_i)| \]
for each $i$.

\end{defn}

The condition that $\phi$ has remnant is slightly
weaker than saying that $\phi(\bar G)$ is Nielsen reduced (see
e.g.\ \cite{ls77}), which would make some additional assumptions on
the word length of the remnant subwords.

Throughout the rest of the paper, we will fix a particular generating
set $\bar G = \{g_1, \dots, g_n\}$ for $G$. 
Given two homomorphisms $\phi, \psi: G \to H$, there is a
homomorphism $\phi * \psi: G * G \to H$, defined as follows: Denoting $G =
\langle g_1, \dots, g_n \rangle$, we write $G*G = \langle g_1, \dots, g_n,
g_1', \dots, g_n'\rangle$. Then we define $\phi * \psi$ on the
generators of $G*G$ by $\phi*\psi(g_i) = \phi(g_i)$ and
$\phi*\psi(g_i') = \psi(g_i)$. 

We have:
\begin{lem}\label{trivialcoin}
If $\phi * \psi: G*G \to H$ has remnant, then $\phi(G) \cap \psi(G) =
\{1\}$. In particular this means that $\Eq(\phi, \psi) = \{1\}$. 
\end{lem}
\begin{proof}
If $\phi(G) \cap \psi(G)$ contains some nontrivial element, then we
have $x, y \in G$, both nontrivial, with 
\[ \phi(x)\psi(y)^{-1} = 1, \]
which is not possible when $\phi*\psi$ has remnant: writing $x$ and
$y$ in terms of generators will show that the left side above cannot
fully cancel.
\end{proof}

The two lemmas can be used to compute doubly-twisted conjugacy classes as in the following example:
\begin{exl}
We will examine doubly-twisted conjugacy classes of the homorphisms $\phi, \psi: G \to G$ with $G= \langle a, b \rangle$ defined by:
\[ \phi: \begin{array}{rcl}
a & \mapsto & aba \\
b & \mapsto & b^{-1}a \end{array}
\quad 
\psi: \begin{array}{rcl}
a & \mapsto & b^2a^{-1} \\
b & \mapsto & a^3 \end{array}
 \]
Given words $u, v \in G$, let $\eta_{(u,v)} = \hat \phi_u^v * \hat
\psi$. Our two lemmas together show that for $u,v \in G$, we will have
$[u] \neq [v]$ whenever $\eta_{(u,v)} : \hat G * \hat G \to \hat G$
has remnant. This remnant condition can be easily checked by hand. We
will apply this strategy for all words $u,v$ of length 0 or 1.

The homomorphism $\eta = \eta_{(1,b)}$ is as follows:
\[ \eta = \begin{array}{rcl} 
a &\mapsto & b^{-1}abab \\
b & \mapsto & b^{-2}ab \\
z &\mapsto & b^{-1}zb \\
a' & \mapsto & b^2a^{-1} \\
b' & \mapsto & a^3 \\
z' & \mapsto & z
\end{array} \]
We can see that $\eta$ has remnant, and thus by Lemma \ref{trivialcoin} that $\Eq(\hat
\phi_1^{b}, \hat \psi) = 1$, and thus by Lemma \ref{doublemv}
that $[1] \neq [b]$. 

Checking the appropriate homomorphisms (due to the asymmetry in the
role of $u$ and $v$, an unsuccessful check for the pair $(u,v)$ might
actually succeed for the pair $(v,u)$) we obtain the following additional inequalities: 
\[ [a]\neq [1],\quad [a]\neq [b],\quad [a^{-1}] \neq [1], \quad
[a^{-1}] \neq [b]. \]
This method is somewhat tedious to perform by hand. A web-based
computer implementation of the process is available for testing at the
author's website.\footnote{The technique is implemented in GAP, with a
  web-based frontend. The front-end and GAP source code are available
 at \url{http://faculty.fairfield.edu/cstaecker}.}
\end{exl}

The checks for remnant in the above example are equivalent to a
related noncancellation condition:

\begin{thm}\label{classdist}
Let $u, v \in H$ be distinct words. If $\phi^v * \psi$ has
remnant, and if, for each generator $g$ of $G*G$, the remnant words
$\Rem_{\phi^v * \psi}(g)$ do not fully cancel in any product of the form 
\begin{equation}\label{products}
(\phi^v *\psi (g)) v^{-1}u, \quad u^{-1}v (\phi^v*\psi(g)),
\end{equation}
then $[u] \neq [v]$.
\end{thm}
\begin{proof}
Let $\hat G = G * \langle z \rangle$ and $\hat H = H * \langle z
\rangle$, and let $\hat\phi_u, \hat\psi: \hat G \to \hat H$ be defined as in Lemma
\ref{doublemv}. We will show that $\hat\phi_u^v * \hat\psi$ has
remnant. 

For brevity, let $\eta = \hat\phi_u^v * \hat\psi$, and let us denote
the generators of the free product so that $\hat G * \hat G = \langle g_1,
\dots, g_n, z, g_1', \dots, g_n', z' \rangle$. Then the homomorphism
$\eta$ is given by:  
\[ \eta = \hat\phi_u^v * \hat\psi: \begin{array}{rcl} 
g_i &\mapsto & \phi^v(g_i) \\
z &\mapsto & v^{-1}uzu^{-1}v \\
g'_i & \mapsto & \psi(g_i) \\ 
z' & \mapsto & z \end{array} \]

To show that $\eta$ has remnant, we must show that the words
$\eta(g_i)$ have noncancelling subwords in various products of the
form in Definition \ref{remnantdef}. For each $i$, let $w_i$ be the
subword of the remnant of $\Rem_{\phi^v*\psi}(g_i)$ which has no
cancellation in any of the products in \eqref{products}. Similarly let
$w_i'$ be the subword of the remnant of $\Rem_{\phi^v*\psi}(g_i')$
with no cancellation in any of the products in \eqref{products}. 

Let us first examine subwords of $\eta(g_i)$ in products of the
form 
\[ \eta(g_i)\eta(g_j)^{\pm 1} = \phi^v(g_i)\phi^v(g_j)^{\pm 1} \text{\quad
  or \quad} \eta(g_i)\eta(g_j')^{\pm 1} = \phi^v(g_i) \psi(g_j)^{\pm 1}. \]
In these products, $w_i$ will 
remain uncancelled (unless $j=i$ with exponent $-1$) because it is a
subword of $\Rem_{\phi^v*\psi}(g_i)$.

Now consider
\[ \eta(g_i)\eta(z)^{\pm 1} = \phi^v(g_i) v^{-1}u z^{\pm 1} u^{-1}v. \]
Here $w_i$ will remain uncancelled by the hypotheses on products of the
form in \eqref{products}, together with the fact that no cancellation
can occur with the $z^{\pm 1}$ because $u$ and $v$ do not use the
letter $z$. Finally we must consider products of the form
\[ \eta(g_i)\eta(z')^{\pm 1} = \phi^v(g_i)z, \]
in which clearly $w_i$ does not cancel. 

We have shown that $w_i$ has no cancellation in products of the form
in Definition \ref{remnantdef} involving
$\eta(g_i)$ on the left. Similar arguments will show that $w_i$ has no
cancellation in products involving $\eta(g_i)$ on the right.
Identical arguments will show that the
words $w'_i$ are uncancelled in various products involving
$\eta(g_i')$, and thus $\hat\phi_u^v*\hat\psi$ has remnant.
Since $\hat\phi_u^v * \hat\psi$ has remnant, we have $\Eq(
\hat\phi_u^v, \hat\psi) = \{1\}$ by Lemma \ref{trivialcoin}, and thus $[u]
\neq [v]$ by Lemma \ref{doublemv}.
\end{proof}

Note that Theorem \ref{classdist} cannot be used in fixed point
theory to distinguish singly-twisted conjugacy classes, since
$\phi^v * \id$ can never have remnant. 

\section{Generic properties}\label{genericsection}

A theorem of Robert F.\ Brown in \cite{wagn99} shows that ``most''
homomorphisms have remnant. Theorem 3.7 of that paper is:
\begin{lem}\label{rfblem}
Let $G$ be a free group with generators $g_1, \dots, g_n$ with $n>1$.
Given any $\epsilon > 0$, there exists some $M>0$ such that, if $\phi:
G \to G$ is an endomorphism chosen at random with $|\phi(g_i)|\le M$ for
all generators $g_i \in G$, then the probability that $\phi$ has
remnant is greater than $1-\epsilon$.
\end{lem}

The above is the only result of its kind typically referenced in the
Nielsen theory literature, but it is in the spirit of a well
established theory of generic group properties. (See \cite{olli05} for
a survey.)

For a free group $G$ and a natural number $p$, let $G_p$ be the subset
of all words of length at most $p$. For a subset $S \subset G$, let
$S_p = S \cap G_p$. The \emph{asymptotic density} (or simply
\emph{density}) of $S$ is defined as 
\[ D(S) = \lim_{p \to \infty} \frac{|S_p|}{|G_p|}, \]
where $|\cdot|$ denotes the cardinality. The set $S$ is said to be
\emph{generic} if $D(S) = 1$.

Similarly, if $S \subset G^l$ is a set of $l$-tuples of
elements of $G$, the asymptotic density of $S$ is defined as
\[ D(S) = \lim_{p \to \infty} \frac{|S_p|}{|(G_p)^l|}, \]
where $S_p = S \cap (G_p)^l$, and
$S$ is called \emph{generic} if $D(S) = 1$.

A homomorphism on the free group $G = \langle g_1, \dots, g_n\rangle$
is equivalent combinatorially to an $n$-tuple of elements of $G$
(the $n$ elements are the words $\phi(g_i)$ for each
generator $g_i$). Thus the asymptotic density of a set of
homorphisms can be defined in the same sense as above, viewing the set of
homomorphisms as a collection of $n$-tuples. The statement of Lemma
\ref{rfblem}, then, is simply that the set of endomorphisms $G \to
G$ with remnant is generic. Similarly we can define the density of a set
of pairs of homorphisms by viewing it as a collection
of $2n$-tuples (a pair of homomorphisms is equivalent to a pair of
$n$-tuples). 

The statement of Lemma \ref{rfblem} can be strengthened
and extended easily to general homomorphisms (possibly
non-endomorphisms) using a generic property from \cite{ao96}.
Consider the setting of homomorphisms $G \to H$, where
$G$ and $H$ are finitely generated free and $H$ has more than one
generator. Lemma 3 of \cite{ao96} 
implies that for any $l$, the collection of $l$-tuples of $H$ which are
Nielsen reduced (when viewed as sets of elements of $H$) is
generic. This directly gives 
\begin{lem} If the rank of $H$ is greater than 1, then
the set of homomorphisms $G \to H$ with remnant is generic.
\end{lem}

Applying the above to homomorphisms $G * G \to H$ and applying Lemma
\ref{trivialcoin} gives an interesting corollary:
\begin{cor}
If the rank of $H$ is greater than 1, then the set of pairs of
homomorphisms $\phi,\psi: G\to H$ with $\Eq(\phi, \psi) = \{1\}$ is
generic.
\end{cor}
This gives a somewhat counterintuitive result: If $F_i$ is the free
group of rank $i$, a pair of homomorphisms from $F_{1000}$ to $F_2$
will generically have images whose intersection is trivial. (See
results of a similar spirit in \cite{mtv08}, e.g.\ that homomorphisms
of free groups are generically injective but not surjective.)

The cited result from \cite{ao96} is quite a bit stronger: it is shown that
the set of subsets of $G$ having small cancellation property $C'(\lambda)$
is generic for any $\lambda > 0$. This gives stronger results
concerning remnant properties of generic homomorphisms:
\begin{lem}\label{genericlem} Let the rank of $H$ be greater than 1. Then:
\begin{itemize}
\item For any natural number $l$, the set of homomorphisms
  $\phi:G \to H$ with remnant length $l$ is generic.
\item For any $r \in (0,1)$, the set of homomorphisms $\phi: G \to H$
  with remnant ratio $r$ is generic.
\end{itemize}
\end{lem}

We include the above lemma for the sake of completeness, but we will
not actually require its full strength in order to prove our generic
property for doubly-twisted conjugacy.

\begin{thm}
Let $G$ and $H$ be free groups, with the rank of $H$ greater than 1. 
Then the set
\[ S = \{ (\phi,\psi, u,v) \mid [u]\neq [v] \} \]
is generic.
\end{thm}
\begin{proof}
We will slightly extend our free-product notation for homomorphisms as
follows: for a homomorphism $\phi: G\to H$ and a word $w \in H$, let
$\langle x \rangle$ be the free group generated by some new letter
$x$. Then define $\phi * w: G * \langle x \rangle
\to H$ by $\phi*w (g_i) = \phi(g_i)$ for $g_i$ a generator of $G$, and
$\phi*w(x) = w$.

By Theorem \ref{classdist}, $S$ contains all tuples $(\phi, \psi, u,
v)$ such that $\phi^v * \psi * uv^{-1}$ has remnant. This remnant
condition will be satisfied when the bracketed words below have
subwords which do not cancel in any of the following products (except
those which are trivial):
\newcommand{\formy}[1]{[#1]}
\[
\begin{array}{ccc}
\formy{\phi^v(g_i)}(\phi^v(g_j))^{\pm 1},
\quad  &\formy{\phi^v(g_i)}\psi(g_j)^{\pm1},
\quad  &\formy{\phi^v(g_i)}(uv^{-1})^{\pm1}, \\
(\phi^v(g_j))^{\pm 1} \formy{\phi^v(g_i)},
  &\psi(g_j)^{\pm1}\formy{\phi^v(g_i)},
  &(uv^{-1})^{\pm1}\formy{\phi^v(g_i)},\\
\formy{\psi(g_i)}(\phi^v(g_j))^{\pm 1},
&\formy{\psi(g_i)}\psi(g_j)^{\pm1},
&\formy{\psi(g_i)}(uv^{-1})^{\pm1}, \\
(\phi^v(g_j))^{\pm 1}\formy{\psi(g_i)},
&\psi(g_j)^{\pm1}\formy{\psi(g_i)},
&(uv^{-1})^{\pm1}\formy{\psi(g_i)}, \\
\formy{uv^{-1}}(\phi^v(g_j))^{\pm 1},
&\formy{uv^{-1}}\psi(g_j)^{\pm1}.
&
\end{array}
\]

It can be verified that there will be noncanceling subwords in
the bracketed parts above when $\phi*\psi*u*v$ has remnant.
We will verify the first and last of these: 

In the product
\[ \formy{\phi^v(g_i)} (\phi^v(g_j))^{\pm 1} =
   [v^{-1}\phi(g_i)v]v^{-1}\phi(g_j)^{\pm 1}v, \]
if $\phi*\psi*u*v$ has remnant, then a portion of $\phi(g_i)$ will
remain uncanceled.
Now consider the product:
\[ \formy{uv^{-1}} \psi(g_j)^{\pm1}. \]
Again, if $\phi*\psi*u*v$ has remnant, then a portion of $u$ will
remain uncanceled. It is easy
to check that various remnant words of $\phi*\psi*u*v$ are similarly uncanceled
in the other of the 14 products above.

Thus $S$ contains the set of tuples $(\phi, \psi, u, v)$ such that
$\phi*\psi*u*v$ has remnant. But this set of tuples is generic by
Lemma \ref{genericlem}, since a choice of a tuple with $\phi*\psi*u*v$
having remnant is combinatorially equivalent to choosing a single
homomorphism $\eta: F_{2n+2} \to H$, where $F_{2n+2}$ is the free
group on $2n+2$ generators. Since $S$ contains a generic set, it is
itself generic.
\end{proof}

\end{document}